\documentclass[11pt,reqno]{amsart}
\usepackage{amsthm,amssymb,amsmath}
\usepackage{textcomp}
\usepackage[foot]{amsaddr}
\usepackage{xcolor}
\usepackage{hyperref}
\usepackage[T1]{fontenc}
\usepackage{lmodern}
\hypersetup{
	colorlinks =true , 
	linkcolor =blue,
	urlcolor = magenta,
	citecolor =red}

\newtheorem{theorem}{Theorem}[section]
\newtheorem*{theorem*}{Theorem}
\newtheorem{lemma}{Lemma}[section]

\theoremstyle{definition}

\newtheorem{definition}{\bf Definition}[section]
\newtheorem*{definition*}{\bf Definition}

\newtheorem{example}{\bf Example}[section]

\newtheorem{remark}{Remark}
\newtheorem*{remark*}{Remark}

\newtheorem*{example*}{\bf Example}

\begin{document}

	\title[Semilinear Poisson equations with asymptotically linear nonlinearities]{Existence result for a system of two semilinear coupled Poisson equations with asymptotically linear nonlinearities}

	\author{Ablanvi Songo}
	\address{D\'{e}partement de math\'{e}matiques} 
	\address{Universit\'{e} de Sherbrooke, Sherbrooke, 2500 boul. de l'université, J1K 2R1, QC, Canada}
	\email{\href{mailto:ablanvi.songo@usherbrooke.ca}{{\textcolor{blue}{ablanvi.songo@usherbrooke.ca}}}}
		
	\noindent
	

	\begin{abstract}
	We establish the existence of at least one solution to a system of two semilinear coupled Poisson equations with asymptotically linear nonlinearities, without imposing the Ambrosetti-Rabinowitz condition or any of its refinements. The proof relies on a generalized saddle point theorem due to Colin and the author \cite{CS}.
	\end{abstract}

	\maketitle

	\section{Introduction}
	In this paper, we study the following system of two semilinear coupled Poisson equations 
	\begin{equation}
		\label{eq 1}
		\begin{cases}
			-\Delta u +u = g_1(x)f_1(v),\; x \in \mathbb{R}^N,\\
			-\Delta v +v= g_2(x) f_2(u), \;x \in \mathbb{R}^N, \\
			u,v \in H^{1}(\mathbb{R}^N).
		\end{cases}
	\end{equation}
	
	We are interested in the asymptotically linear case
	\begin{equation}
		\label{eq 2}
		\lim\limits_{|v|\to \infty}\dfrac{f_1(v)}{v}= \lim\limits_{|u|\to \infty}\dfrac{f_2(u)}{u}=0.
	\end{equation}
	
	If, we define in $H^{1}(\mathbb{R}^N)\times H^{1}(\mathbb{R}^N)$ the functional
	\begin{equation}
		\label{eq 3}
		J(u,v)= \int_{\mathbb{R}^N} \Big(\nabla u\nabla v +uv\Big)dx -\int_{\mathbb{R}^N} \Big(g_1(x)F_1(v)+g_2(x)F_2(u)\Big)dx,
	\end{equation}
	where
	\begin{equation*}
		F_1(v)=\int_{0}^{v(x)}f_1(\zeta)d\zeta, \quad F_2(u)=\int_{0}^{u(x)}f_2(\zeta)d\zeta,
	\end{equation*}
	then, it is well known that the first part of the functional $J$ is strongly indefinite operator and under suitable assumptions, $J$ is of class $\mathcal{C}^1$ on $H^{1}(\mathbb{R}^N)\times H^{1}(\mathbb{R}^N)$, and its critical points are weak solutions of $(\ref{eq 1})$ . Precisely, the functional $J$ has the form
	\begin{equation*}
		\dfrac{\|Q(u,v)\|^2}{2}- \dfrac{\|P(u,v)\|^2}{2} -\varphi(x,u,v), 	\end{equation*}
	for every $(u,v)\in \Big( X= Y\oplus Z:= H^{1}(\mathbb{R}^N)\times H^{1}(\mathbb{R}^N), \|\cdot \|\Big)$, where $Y$ and $Z$ are infinite-dimensional subspaces of $X$, $P$ and $Q$ are, respectively, the orthogonal projections of $X$ into $Y$ and $Z$.\\
	
	There is a vast literature devoted to the study of the existence of solutions to semilinear coupled Poisson equations, both in the whole space and in bounded or unbounded domains; see for example \cite{ FDR, SC, CF, ZL, FY, LY} and the references therein. In \cite{SC, CF, ZL, LY}, the following condition on the nonlinearities near $0$ is assumed:
	\begin{equation*}
		\lim\limits_{v\to 0}\dfrac{f_1(v)}{v}= \lim\limits_{u\to 0}\dfrac{f_2(u)}{u}=0
	\end{equation*}
	and the authors obtained nontrivial solutions by means of variational arguments in the spirit of the generalized linking theorem of Kryszewski and Szulkin \cite{KS}. \\
	
	To the best of our knowledge, no results are available in the literature concerning the existence of solutions to semilinear coupled Poisson equations of the form $(\ref{eq 1})$ when the nonlinearities $f_1, f_2$ satisfy only the asymptotically linear condition $(\ref{eq 2})$.\\
	
	The aim of this paper is to prove the existence of at least one solution to problem $(\ref{eq 1})$. The main difficulty arises from the fact that the energy functional $(\ref{eq 3})$ associated with $(\ref{eq 1})$ possesses a strongly indefinite quadratic part. Consequently, the proofs of our main results cannot rely on classical min–max arguments. It is therefore of particular interest to investigate this strongly indefinite case. By applying a generalized saddle point theorem due to the author and Colin \cite{CS}, we establish an existence result for problem $(\ref{eq 1})$.\\
	
	The paper is organized as follows. In Section 2, we recall some classical results and present additional preliminary material needed in the sequel. In Section 3, we state and prove our main result.
	\section{Kryszewski-Szulkin degree theory}
	In this section, we are following the presentation of the degree theory of Kryszewski and Szulkin given in \cite{Wi} by Michel Willem.
	
	Let $Y$ be a real separable Hilbert space endowed with inner product $( \cdot, \cdot )$ and the associated norm $\|\cdot\|$.  
	
	On $Y $ we consider the $\sigma-$topology introduced by Kryszewski and Szulkin; that is, the topology generated by the norm 
	\begin{equation}
		|u|_\sigma := \sum_{k =0}^{\infty} \frac{1}{2^ {k+1}}|(u, e_k)|, \quad u \in Y,
	\end{equation}
	where $(e_k)_{k\ge 0}$ is a total orthonormal sequence in $Y$.\\
	
	\begin{remark}	
		\label{remark 1}
		By the Cauchy-Schwarz inequality, one can show that $|u|_\sigma \le \|u\|$ for every $u \in Y$. Moreover, if $(u_n)$ is a bounded sequence in $Y$ then 
		\begin{equation*}
			u_n \rightharpoonup u \Longleftrightarrow u_n \overset{\sigma}{\rightarrow} u,
		\end{equation*}
		where $\rightharpoonup$ denotes the weak convergence and $ \overset{\sigma}{\rightarrow}$ denotes the convergence in the $\sigma-$topology.
	\end{remark}
	Let $U$ be an open bounded subset of $Y$ such that its closure $\overline{U}$ is $\sigma-$closed.
	\begin{definition}[\cite{Wi}]
		\label{definition 2.1}
		A map $f : \overline{U} \rightarrow Y$ is $\sigma-$admissible (admissible for short) if\\
		(1) $f$ is $\sigma-$continuous, \\
		(2) each point $u \in U$ has a $\sigma-$neighborhood $N_u$ in $Y$ such that $(\operatorname{id}-f)(N_u\cap U)$ is contained in a finite-dimensional subspace of $Y$.
	\end{definition}
	\begin{definition}[\cite{Wi}]
		A map $h :[0,1]\times \overline{U}\rightarrow Y$ is an admissible homotopy if \\
		(1)  $0 \notin h([0,1] \times \partial U)$,\\
		(2) $h$ is $\sigma-$continuous, that is $t_n \rightarrow t$ and $u_n \overset{\sigma}{\rightarrow} u$ implies $h(t_n,u_n) \overset{\sigma}{\rightarrow} h(t,u)$,\\
		(3) $h$ is $\sigma-$locally finite-dimensional. That is, for any $(t,u) \in[0,1]\times U$ there is a neighborhood $N_{(t,u)}$ in the product topology of $[0,1]$ and $(X, \sigma)$ such that $\{ v-h(s,v) \; | \; (s,v) \in N_{(t,u)} \cap ([0,1]\times U)\}$ is contained in a finite-dimensional subspace of $Y$.
	\end{definition} Then, for an admissible map $f$ such that $0 \notin f(\partial U)$   we have (see \cite[ Theorem 6.6]{Wi})
	\begin{equation*}
		deg \Big(h(0,.), U\Big) = deg\Big(h(1,.),U\Big),
	\end{equation*}
	where $deg$ is the topological degree of $f$ (about 0). Such a degree possesses the usual properties; in particular, if $f : \overline{U} \rightarrow Y$ is admissible with $0 \notin f(\partial U)$ and $deg\Big(f, U\Big)\ne 0$, then there exists $u \in U$ such that $f(u)= 0$.\\

	Now, let $X = X^{-}\oplus X^{+}$, where $X^-$ is closed and $X^+=(X^-)^\perp$, be a real separable Hilbert space endowed with the inner product $\langle \cdot, \cdot \rangle$ and the associated norm $\|\cdot\|$. Let $(e_k)_{k\ge 0}$ be an orthogonal basis of $X^-$. We set on $X$, a new norm defined by
	\begin{equation*}
		|u|_\tau := \max \Bigg(\sum_{k =0}^{\infty} \frac{1}{2^ {k+1}}|\langle Pu, e_k\rangle|, \|Qu\| \Bigg) \quad u \in X.
	\end{equation*}
	Here, $P$ and $Q$ denote the orthogonal projections of $X$ onto $X^-$ and $X^+$, respectively; $\tau$ denotes the topology generated by this norm. The topology $\tau$ was introduced by Kryszewski and Szulkin \cite{KS}.
	\begin{remark}	
		\label{remark 2} For every $u\in X$, 
		we have $\|Qu\| \le \|u\|_\tau$ and $|Pu|_{\sigma} \le \|u\|_\tau$. Moreover, if $(u_n)$ is a bounded sequence in $X$ then 
		\begin{equation}
			\label{eq 5}
			u_n \overset{\tau}{\rightarrow} u \Longleftrightarrow	Pu_n \rightharpoonup Pu \quad\text{and}\quad Qu_n \to Qu,
		\end{equation}
		where $\rightharpoonup$ denotes the weak convergence and $ \overset{\tau}{\rightarrow}$ denotes the convergence in the $\tau-$topology.
	\end{remark}
	Let $B_r$ denote the open ball in $X$ of radius $r>0$ about $0$ and let $\partial B_r$ denote its boundary.\\
	\begin{definition}
		Let $J \in \mathcal{C}^1(X,\mathbb{R})$. \\
		\begin{enumerate}
			\item We say that the functional $J$ is $\tau-$upper semi-continuous if  $u_n \overset{\tau}{\rightarrow} u$ implies
			\begin{equation*}
				J(u)\ge \overline{\underset{n\to \infty}{\lim}} J(u_n).
			\end{equation*}
			\item  We say that $\nabla J $ is weakly sequentially continuous if \[u_n \rightharpoonup u\quad \text{implies}\quad \nabla J(u_n) \rightharpoonup \nabla J(u).\]
			\item 
			The functional $J$ is said to satisfy the $(PS)_c$ condition (or the Palais-Smale condition at level $c$) if any sequence $(u_n)\subset X$ such that 
			\begin{equation*}
				J(u_n) \rightarrow c \quad \textit{and} \quad J' (u_n) \rightarrow 0, \quad \text{as}\quad n\to \infty,
			\end{equation*}
			has a convergent subsequence.\\
		\end{enumerate}
	\end{definition}

	We consider the class of $\mathcal{C}^1-$functionals $J : X \rightarrow \mathbb{R}$ such that 
	
	\text{(A)}	$J$ is $\tau-$upper semi-continuous and $\nabla J$ is weakly sequentially continuous.\\

	We recall the following result from \cite{CS} that will play a key role in the proof of our main result; see \cite[Theorem 2.1 ]{CS}.\\
	For $\rho >0$, set 
	\begin{equation}
		\label{eq 6}
		M = \Big\{ u \in X^- \; |\; \|u\| \le \rho \Big\},
	\end{equation}
	Then $M$ is a sub-manifold of $Y$ with boundary $\partial M$.
	
	\begin{theorem}[Generalized Saddle Point Theorem, \cite{CS}]
		\label{theorem 2.1}
		Assume that $J$ satisfies $(A)$, that is, $J$ is $\tau-$upper semi-continuous and $\nabla J$ is weakly sequentially continuous. We suppose that
		\begin{equation}
			\label{eq 7}
			b:= \underset{ u\in X^+}{\inf} \; J(u) > a := \underset{u\in\partial M}{\sup}\; J(u). 
		\end{equation}
	Let $c \in \mathbb{R}$ be characterized as 
		\begin{center}
			$c:= \underset{\gamma \in \Gamma}{\inf} \; \underset{u\in M}{\sup}\; J(\gamma(u))$, \\ 
		where\quad	$\Gamma : = \Big\{\gamma : M \rightarrow X \;\Big|\;\gamma \quad \text{is}\quad \tau-$continuous, $\gamma \big|_{ \partial M} = \operatorname{id}$ and\\ every  $u\in \operatorname{int}(M)$ has a $\tau-$neighborhood $N_u$ in $X$ such that $(\operatorname{id}-\gamma)\Big(N_u\cap \operatorname{int}(M)\Big)$ \\ is contained in a finite-dimensional subspace of $X \Big \}$.
		\end{center}	
	Then, there exists $(u_n)\subset X$ such that 
	\begin{equation*}
		J(u_n) \rightarrow c, \qquad J'(u_n) \rightarrow 0, \quad \text{as}\quad n\to \infty.
	\end{equation*}
	\end{theorem}

	\section{Main result}

	Consider the Hilbert space 
	\begin{equation*}
		H^1(\mathbb{R}^N) := \Big\{u\in L^2(\mathbb{R}^N)\;|\; \nabla u \in L^2(\mathbb{R}^N)\Big\},
	\end{equation*}
	endowed with the inner product 
	\begin{equation*}
		(u,v)_1 = \int_{\mathbb{R}^N} \Big(\nabla u \cdot\nabla v + uv\Big) dx,
	\end{equation*}
	and the associated norm $\|.\|_1$ given by
	\begin{equation*}
		\|u\|_{1}^2 :=	\int_{\mathbb{R}^N} \Big(|\nabla u |^2 + |u|^2\Big) dx.
	\end{equation*}
	
	We denote the product space \[X= H^{1}(\mathbb{R}^N)\times H^{1}(\mathbb{R}^N),\] the Hilbert space endowed with the inner product 
	\begin{equation}
		\langle (u,v), (\phi,\psi)\rangle := \int_{\mathbb{R}^N} \Big( \nabla u \nabla \phi + u\phi \Big) dx + \int_{\mathbb{R}^N} \Big(\nabla v  \nabla \psi +v\psi \Big)dx, \;\; \text{for every}\;\; (\phi, \psi)\in X,
	\end{equation}
	and the corresponding norm $\|\cdot\|$, given by
	\begin{equation*}
		\|(u,v),(u,v)\|^2 = \| u \|_{1}^2 + \|v\|_{1}^2.
	\end{equation*}
	If we define
	\begin{equation*}
		X^- := \Big\{(-v,v) \in X \Big\} \quad \text{and}\quad  X^+:=\Big\{(u,u)\in X\Big\},
	\end{equation*}
	since we can write $(u,v)$ as
	\begin{equation*}
		(u,v) = \dfrac{1}{2}(u+v,u+v)+\dfrac{1}{2}(-v+u, v-u),
	\end{equation*}
	then, $X= X^-\oplus X^+$.\\
	
	Let us denote by $P$ the projection of $X$ onto $X^-$ and by $Q$ the projection of $X$ onto $X^+$. We have (see \cite{ZL} or \cite{SC})
	\begin{equation}
		\label{eq 9}
		\int_{\mathbb{R}^N} \Big( \nabla u \nabla v + uv \Big) dx  =  \dfrac{\|Q(u,v)\|^2}{2} -\dfrac{\|P(u,v)\|^2}{2}.
	\end{equation}
	
	Finally, let us define the functional 
	\begin{eqnarray}
		\label{eq 10}
		J&:& X \rightarrow \mathbb{R}\nonumber\\
		J(u,v) &:=& \int_{\mathbb{R}^N} \Big( \nabla u \nabla v + uv \Big) dx -\int_{\mathbb{R}^N} \Big(g_1(x)F_1(v)+g_2(x)F_2(u)\Big)dx \\
		&=& \dfrac{\|Q(u,v)\|^2}{2} -\dfrac{\|P(u,v)\|^2}{2} - \varphi(x,u,v)\nonumber,
	\end{eqnarray}
	where
	\begin{equation*} 
		\varphi(x,u,v) := \int_{\mathbb{R}^N} \Big(g_1(x)F_1(v)+g_2(x)F_2(u)\Big)dx.
	\end{equation*}
	
	\begin{definition}
		\label{def 3.1}
		We say that $(u,v)$ is a weak solution of problem $(\ref{eq 1})$ if $(u,v)\in X$, and satisfies for any $ (\phi, \psi) \in X$ : 
		\begin{multline*}
		\int_{\mathbb{R}^N} \Big( \nabla u \nabla \psi + u\psi \Big) dx + \int_{\mathbb{R}^N} \Big(\nabla v  \nabla \phi +v\phi \Big)dx -\int_{\mathbb{R}^N} \Big(g_1(x)f_1(v) \psi+g_2(x)f_2(u) \phi\Big)dx =0. \\
	\end{multline*}
	\end{definition}
	
	Throughout this section, $|\cdot|_p$ stands for the $L^p$-norm; $\to$ and $\rightharpoonup$ denote strong and weak convergence, respectively.\\
	
	Our assumptions on $(\ref{eq 1})$ are the following: 
	
	\begin{enumerate}
		\item[$(f)$] The functions $f_1, f_2 \in \mathcal{C}(\mathbb{R}, \mathbb{R})$ satisfy $(\ref{eq 2})$, $tf_1(t)\ge0,\; tf_2(t)\ge 0$ for all $t\in \mathbb{R}$. 
		\item[$(g)$] The functions $g_1, g_2\in L^2(\mathbb{R}^N)\cap  L^{\infty}(\mathbb{R}^N)$, $g_1(x)\ge0$, $g_2(x)\ge 0$ for every $x\in \mathbb{R}^N$.\\
	\end{enumerate}
	\begin{remark}
			We would like to mention that 
		\begin{enumerate}
		\item these conditions coincide with those considered in \cite{LiSh}.
		\item 	since assumption $(f)$ implies that $f_1(0)=0$ and $f_2(0)=0$, then $(u,v)=(0,0)$ is a trivial solution of problem $(\ref{eq 1})$.
		\end{enumerate}
	\end{remark}
	\begin{example}
The set of functions $f_i$ and $g_i$, $i\in \Big\{1, 2 \Big \}$, that satisfy the above conditions $(f)$ and $(g)$ is nonempty. Indeed, define
\[f_1(t)= \dfrac{t}{1+t^2}, \quad f_2(t)=\dfrac{t}{1+t^4}, \qquad \text{for every}\quad t\in \mathbb{R};\]
\[ g_1(x)=g_2(x)= e^{-|x|^2}, \qquad \text{for every}\quad x\in \mathbb{R}^N.\]
Then, the functions $f_i$ and $g_i$, $i\in \Big\{1,2 \Big\}$ satisfy assumptions $(f)$ and $(g)$.
	\end{example}
	
The main result of this paper is as follows:
	
	\begin{theorem}
		\label{theorem 3.1}
		Under assumptions $(f)$ and $(g)$, problem $(\ref{eq 1})$ admits at least one nontrivial solution.
	\end{theorem}

	Under assumptions $(f)$ and $(g)$, we have 
	\begin{equation*}
		g_1(x)F_1(v) \in L^1(\mathbb{R}^N)\;\; \text{and}\;\; g_2(x)F_2(u) \in L^1(\mathbb{R}^N), \quad \text{for every}\;\; (u, v) \in X.
	\end{equation*}
	Therefore, the functional $J(u,v)$ given in $(\ref{eq 10})$ is well defined. Furthermore, using variational standard arguments, the functional $J(u,v)$ is of class $\mathcal{C}^1$ in $X$ and 
	\begin{multline}
		\label{eq 11}
		J'(u,v)(\phi, \psi) = \int_{\mathbb{R}^N} \Big( \nabla u \nabla \psi + u\psi \Big) dx + \int_{\mathbb{R}^N} \Big(\nabla v  \nabla \phi +v\phi \Big)dx\\ -\int_{\mathbb{R}^N} \Big(g_1(x)f_1(v) \psi+g_2(x)f_2(u) \phi\Big)dx, \;\: \text{for every}\;\; (\phi, \psi) \in X. 
	\end{multline}
	
	Hence, the weak solutions of problem $(\ref{eq 1})$ are exactly the critical points of $J(u,v)$ in $X$ (see Definition $\ref{def 3.1}$).\\
	
	To prepare the proof of Theorem $\ref{theorem 3.1}$, we first establish several lemmas. \\
	\begin{lemma}
		\label{lem 3.1}
		The functional $J(u,v)$ given in $(\ref{eq 10})$ satisfies $(A)$, that is, $J(u,v)$ is $\tau-$upper semi-continuous and $\nabla J$ is weakly sequentially continuous.
	\end{lemma}
	\begin{proof}
		
		1. The functional $J$ is  $\tau-$upper semi-continuous :
		for every $c\in \mathbb{R}$, let show that the set
		
		\[\Big\{(u,v)\in X\,|\, J(u,v)\ge c \Big\}\quad \text{is}\quad \tau-\text{closed}.\]
		Let $(u_n,v_n) \subset X$ such that \[(u_n,v_n)\overset{\tau}{\rightarrow} (u,v) \quad \text{in} \quad X \quad \text{and} \quad c\le J(u_n,v_n).\] Then, $(u_n,v_n)$ is bounded. Indeed, by Remark $\ref{remark 2}$, $\Big(\|Q(u_n,v_n)\|\Big)$ is bounded;  since 
		\begin{equation*}
			\varphi(x,u,v) = \int_{\mathbb{R}^N} \Big(g_1(x)F_1(v)+g_2(x)F_2(u)\Big)dx \ge 0,
		\end{equation*} 
		then
		\begin{eqnarray*}
			\|P(u_n,v_n)\|^2 &=& \|Q(u_n,v_n)\|^2 -2 J(u_n,v_n) -2 \varphi(x,u_n,v_n)\\
			&\le& \|Q(u_n,v_n)\|^2 -2c.
		\end{eqnarray*}	
		
		It follows that $\Big(\|P(u_n,v_n)\|\Big)$ is also bounded. 
		
		Thus, there is a subsequence of $(u_n,v_n)$ that we still denote $(u_n, v_n)$ such that  
		\[(u_n,v_n) \rightharpoonup (u,v)\quad \text{in}\quad X.\]
		 Up to a subsequence, we also have \begin{equation*}
			(u_n,v_n)\to (u,v),\;\; \text{almost everywhere in}\;\;\mathbb{R}^N.
		\end{equation*}
		
		Since $g_1(x)F_1(v_n)\ge 0$ and $g_2(x)F_2(u_{n})\ge 0$, then by the Fatou Lemma, we have 
		\begin{equation*}
			\int_{\mathbb{R}^N} g_1(x)F_1(v) dx= \int_{\mathbb{R}^N} \underset{n\rightarrow \infty}{\underline{\lim}}g_1(x)F_1(v_{n})dx  \le \underset{n\rightarrow \infty}{\underline{\lim}} \int_{\mathbb{R}^N} g_1(x)F_1(v_{n})dx.
		\end{equation*}
		In the same way, we have 
		\begin{equation*}
			\int_{\mathbb{R}^N} g_2(x) F_2(u) dx= \int_{\mathbb{R}^N} \underset{n\rightarrow \infty}{\underline{\lim}} g_2(x)F_2(u_{n})dx  \le \underset{n\rightarrow \infty}{\underline{\lim}} \int_{\mathbb{R}^N}g_2(x) F_2(u_{n})dx.
		\end{equation*}
		Since $\|\cdot\|$ is weak lower semi-continuous, we have 
		\begin{equation*}
			\|P(u,v)\|^2 \le \underset{n\rightarrow \infty}{\underline{\lim}} \|P(u_n,v_n)\|^2.
		\end{equation*}
		Moreover, since \[\|Q(u_n,v_n)\|^2 \to \|Q(u,v)\|^2,\quad \text{as}\quad n\to \infty, \] we have 
		\begin{equation*}
		 \underset{n\rightarrow \infty}{\overline{\lim}} \Big( -\|Q(u_n,v_n)\|^2\Big)= 
			-\|Q(u,v)\|^2 = 	\underset{n\rightarrow \infty}{\underline{\lim}} \Big(-\|Q(u_n,v_n)\|^2\Big).
		\end{equation*}
		Hence, 
		\begin{eqnarray*}
			-J(u,v) &=& \dfrac{1}{2}\Big( \|P(u,v)\|^2-\|Q(u,v)\|^2 \Big) +\varphi (x,u,v)\\ &\le& \underset{n\rightarrow \infty}{\underline{\lim}} \Bigg(\dfrac{1}{2}\Big( \|P(u_n,v_n)\|^2-\|Q(u_n,v_n)\|^2 \Big) +\varphi (x,u_n,v_n)  \Bigg)\\
			&=&  \underset{n\rightarrow \infty}{\underline{\lim}} (-J(u_n,v_n) )\\
			&=& - \underset{n\rightarrow \infty}{\overline{\lim}} J(u_n,v_n)\\
			&\le& -c.
		\end{eqnarray*}
		2. Now, let us show that $\nabla J$ is weakly sequentially continuous. Let $(u_n,v_n) \subset X$ such that \[(u_n,v_n) \rightharpoonup (u,v) \quad \text{in}\quad X.\] We have to show that \[\nabla J(u_n,v_n) \rightharpoonup \nabla J(u,v).\]
		For every $(w,z)\in \mathcal{C}_c^\infty(\mathbb{R}^N)\times \mathcal{C}_c^\infty(\mathbb{R}^N) $, we have 
		
		\begin{equation*}
			\int_{\mathbb{R}^N}\Big(\nabla u_n \nabla z +u_n z +\nabla v_n w +v_n w\Big)dx \to \int_{\mathbb{R}^N}\Big(\nabla u\nabla z +uz +\nabla v \nabla w +vw\Big)dx.
		\end{equation*}
		For every compact subset $\Omega\subset \mathbb{R}^N$ such that $supp \,\{(w,z)\}\subset \Omega \times \Omega$, we have
		\[
		\int_{\mathbb{R}^N}\Big( g_1(x)f_1(v_n) z+g_2(x)f_2(u_n)w\Big)dx = \int_{\Omega}\Big( g_1(x)f_1(v_n) z+g_2(x)f_2(u_n)w\Big)dx.
		\]
		By compact Sobolev embedding theorem, up to a subsequence, since \[(u_n,v_n) \rightharpoonup (u,v)\quad \text{in}\quad \Omega\times \Omega,\] we have
		\begin{equation*}
			(u_n,v_n) \to (u,v) \quad \text{in}\quad L^2(\Omega)\times L^2(\Omega), \quad \text{as}\quad n \to \infty, 
		\end{equation*}
		
		\begin{equation*}
			\Big(u_n(x), v_n(x)\Big) \to \Big(u(x),v(x)\Big), \quad \text{almost everywhere}, \quad \text{in}\quad \Omega\times \Omega, \quad \text{as}\quad n \to \infty.
		\end{equation*}

		The functions $f_i$, $i\in \{1,2\}$ are continuous. So, 
		\begin{equation*}
			f_1(v_n(x))\to f_1(v(x)), \quad \text{almost everywhere}, \quad \text{in}\quad \Omega, \quad \text{as}\quad n \to \infty,
		\end{equation*}
		\begin{equation*}
			f_2(u_n(x))\to f_2(u(x)), \quad \text{almost everywhere}, \quad \text{in}\quad \Omega, \quad \text{as}\quad n \to \infty.
		\end{equation*}
		
		Moreover, by $(\ref{eq 2})$, there exist $c>0$ and $c_0>0$ such that \[|f_i(\zeta)|\le \dfrac{c}{|g_i|_\infty}|\zeta|+c_0,\quad \text{for every} \quad \zeta \in H^1(\mathbb{R}^N).\] It follows that \[g_1(x)f_1(v_n)z\quad\text{and} \quad g_2(x)f_2(u_n)w \quad \text{are in}\quad L^1(\Omega).\] 
		
		The Lebesgue dominated convergence theorem implies that
		\begin{equation*}
			\int_{\Omega}\Big( g_1(x)f_1(v_n) z+g_2(x)f_2(u_n)w\Big)dx \to \int_{\Omega}\Big( g_1(x)f_1(v) z+g_2(x)f_2(u)w\Big)dx, \quad \text{as}\quad n\to \infty.
		\end{equation*}
		
		Thus, by $(\ref{eq 11})$, we have 
		\begin{equation*}
			\Big| J'(u_n,v_n) (w,z)- J'(u,v)(w,z)  \Big| \quad \rightarrow 0, \quad \text{as}\quad n \rightarrow \infty.
		\end{equation*}
		
		Therefore, for every $(w,z)\in \mathcal{C}_c^\infty(\mathbb{R}^N)\times \mathcal{C}_c^\infty(\mathbb{R}^N) $,
		\begin{equation*}
			\langle \nabla J(u_n,v_n), (w,z) \rangle \to \langle \nabla J(u,v), (w,z) \rangle.
		\end{equation*}
		
		Hence, $\nabla J(u_n,v_n) \rightharpoonup \nabla J(u,v)$.
	\end{proof}
	
	The next result shows that the functional $J$ given in $(\ref{eq 10})$ satisfies the geometric assumption $(\ref{eq 7})$ of Theorem $\ref{theorem 2.1}$.
	
	\begin{lemma}
		\label{lem 3.2}
		There exists $\rho>0$ such that 
		$b:= \underset{(u,u)\in X^+}{\inf} \; J(u,u) > a := \underset{(-u,u)\in\partial M}{\sup}\; J(-u,u) $.
	\end{lemma}
	\begin{proof}
		Let $S$ be the optimal constant of Sobolev inequality, that is, 
		\begin{equation}
			\label{eq 12}
			\text{for all}\;\; t \in H^1(\mathbb{R}^N),\; S|t|_2^{2}\le \|t\|_1^{2}.
		\end{equation}
		By $(f)$, there is $c_0>0$ such that for every $(u,v)\in X$,
		\begin{equation}
			\label{eq 13}
			|f_1(v)|\le \dfrac{S}{4|g_1|_\infty}|v|+c_0, \quad |F_1(v)|\le \dfrac{S}{4|g_1|_\infty}|v|^2+c_0|v|,
		\end{equation}
		\begin{equation}
			\label{eq 14}
			|f_2(u)|\le \dfrac{S}{4|g_2|_\infty}|u|+c_0, \quad |F_2(u)|\le \dfrac{S}{4|g_2|_\infty}|u|^2+c_0|u|.
		\end{equation}
		\\
		On $X^+$, by $(\ref{eq 12})$ and Hölder inequality, we have 
		\begin{eqnarray*}
			J(u,u)&=&\dfrac{1}{2}\|Q(u,u)\|^2 -\varphi(x,u,u)\\
			&\ge& \|u\|_1^{2}-\int_{\mathbb{R}^N}|g_1(x)|\Big(\dfrac{Su^2}{4|g_1|_\infty}+c_0|u|\Big)dx -\int_{\mathbb{R}^N}|g_2(x)|\Big(\dfrac{Su^2}{4|g_2|_\infty}+c_0|u|\Big)dx\\
			&\ge& \|u\|_1^{2} - \dfrac{S}{4}|u|_2^2 -c_0 \int_{\mathbb{R}^N}\Big(|g_1(x)|\,|u|\Big)dx -\dfrac{S}{4}|u|_2^2 -c_0\int_{\mathbb{R}^N}\Big(|g_2(x)|\,|u|\Big)dx\\
			&\ge& \|u\|_1^{2} -\dfrac{S}{2}|u|_2^2-c_0|g_1|_2|u|_2 -c_0|g_2|_2|u|_2\\
			&\ge& \|u\|_1^{2} -\dfrac{1}{2}\|u\|_1^{2}-c_0S^{-\frac{1}{2}}\Big(|g_1|_2+|g_2|_2\Big)\|u\|_1\\
			&=&\dfrac{1}{2}\|u\|_1^{2}-c_0S^{-\frac{1}{2}}\Big(|g_1|_2+|g_2|_2\Big)\|u\|_1.
		\end{eqnarray*}
		So, $\underset{ X^+}{\inf}\;J> -\infty$.\\
		
		On $X^-$, since $g_1(x)F_1(v)\ge 0$ and $g_2(x)F_2(u)\ge 0$, we have 
		\begin{eqnarray*}
			J(-u,u) &=& -\dfrac{1}{2}\|P(-u,u)\|^2 -\varphi(x,u,u)\\
			&\le& -\dfrac{1}{2}\|P(-u,u)\|^2 \\
			&=& - \|u\|_1^{2} \\
			&\le& - \|u\|_1^{2} + c_0S^{-\frac{1}{2}}\Big(|g_1|_2+|g_2|_2\Big)\|u\|_1.
		\end{eqnarray*}
	So, for $\rho$ large enough, we obtain the result.
	\end{proof}

	\subsection*{Proof of Theorem $\ref{theorem 3.1}$} The functional $J$ given in $(\ref{eq 10})$, associated with problem $(\ref{eq 1})$, is of $\mathcal{C}^1$ , $\tau-$upper semi-continuous, and its gradient $\nabla J$ is weakly sequentially continuous by Lemma $\ref{lem 3.1}$. Hence, $J$ satisfies assumption $(A)$ of Theorem $\ref{theorem 2.1}$. Moreover, the geometric assumption $(\ref{eq 7})$ of Theorem $\ref{theorem 2.1}$ is fulfilled by Lemma $\ref{lem 3.2}$.\\
	By Theorem $\ref{theorem 2.1}$, for some $c\in \mathbb{R}$, there exists $\Big\{(u_n,v_n)\Big\}\subset X$ such that 
	\begin{equation*}
		J(u_n,v_n) \rightarrow c, \qquad J'(u_n,v_n) \rightarrow 0, \quad \text{as}\quad n\to \infty.
	\end{equation*}
	For $u\in X$, we denote by $u^{\pm}$ the orthogonal projection of $u$ on $X^\pm$.\\
	Let \[\Big(\overline{u}_n,\overline{v}_n\Big) = \Big(u_n^+-u_n^-, v_n^+-v_n^-\Big).\] Then,
	\begin{eqnarray*}
		\|(\overline{u}_n,\overline{v}_n)\|^2 &=&\| (u_n^+-u_n^-, v_n^+-v_n^-)\|^2\\
		&=& \|u_n^+ -u_n^- \|_1^2 + \|v_n^+ -v_n^- \|_1^2\\
		&=& \|u_n^+\|_1^2 +\|u_n^-\|_1^2 + \|v_n^+\|_1^2 +\|v_n^-\|_1^2\\
		&=& \|u_n\|_1^2 +\|v_n\|_1^2\\
		&=& \|(u_n, v_n)\|^2.
	\end{eqnarray*}
	For sufficiently large $n\in \mathbb{N}$, we have $\|J'(u_n,v_n)\| \le 1$. Hence, for such $n$
	\begin{equation*}
		\|(u_n,v_n)\| =\|(\overline{u}_n, \overline{v}_n)\| \ge J'(u_n,v_n)(\overline{u}_n, \overline{v}_n),
	\end{equation*}
	where by $(\ref{eq 11})$,	\begin{multline*}
		J'(u_n,v_n)(\overline{u}_n, \overline{v}_n) = \int_{\mathbb{R}^N} \Big( \nabla u_n \nabla \overline{v}_n + u_n\overline{v}_n \Big) dx + \int_{\mathbb{R}^N} \Big(\nabla v_n  \nabla \overline{u}_n +v_n \overline{u}_n \Big)dx\\ -\int_{\mathbb{R}^N} \Big(g_1(x)f_1(v_n) \overline{v}_n+g_2(x)f_2(u_n)\overline{u}_n\Big)dx. 
	\end{multline*}
	
	On the one hand, we have 
	\begin{multline}
		\label{eq 15}
		\int_{\mathbb{R}^N} \Big( \nabla u_n \nabla \overline{v}_n + u_n\overline{v}_n \Big) dx = \int_{\mathbb{R}^N} \Big( \nabla u_n^+ \nabla v_n^+ + u_n^+ v_n^+ \Big) dx - \int_{\mathbb{R}^N} \Big( \nabla u_n^- \nabla  v_n^- + u_n^- v_n^- \Big) dx \\
		-\int_{\mathbb{R}^N} \Big( \nabla u_n^+ \nabla  v_n^- + u_n^+ v_n^- \Big) dx + \int_{\mathbb{R}^N} \Big( \nabla u_n^- \nabla   v_n^+ + u_n^- v_n^+ \Big) dx,
	\end{multline}
	\begin{multline}
		\label{eq 16}
		\int_{\mathbb{R}^N} \Big(\nabla v_n  \nabla \overline{u}_n +v_n \overline{u}_n \Big)dx = \int_{\mathbb{R}^N} \Big( \nabla v_n^+ \nabla u_n^+ + v_n^+ u_n^+ \Big) dx - \int_{\mathbb{R}^N} \Big( \nabla v_n^- \nabla  u_n^- + v_n^- u_n^- \Big) dx \\
		+\int_{\mathbb{R}^N} \Big( \nabla u_n^+ \nabla  v_n^- + u_n^+ v_n^- \Big) dx - \int_{\mathbb{R}^N} \Big( \nabla u_n^- \nabla   v_n^+ + u_n^- v_n^+ \Big) dx.
	\end{multline}
	Combining equations $(\ref{eq 15})$ and $(\ref{eq 16})$, we get
	\begin{multline}
		\int_{\mathbb{R}^N} \Big( \nabla u_n \nabla \overline{v}_n + u_n\overline{v}_n \Big) dx \\+ \int_{\mathbb{R}^N} \Big(\nabla v_n  \nabla \overline{u}_n +v_n \overline{u}_n \Big)dx = 2 \int_{\mathbb{R}^N} \Big( \nabla u_n^+ \nabla v_n^+ + u_n^+ v_n^+ \Big) dx \\-2 \int_{\mathbb{R}^N} \Big( \nabla u_n^- \nabla  v_n^- + u_n^- v_n^- \Big) dx
	\end{multline}
	But then, by equation $(\ref{eq 9})$, 
	\begin{equation*}
		\int_{\mathbb{R}^N} \Big( \nabla u_n^+ \nabla v_n^+ + u_n^+ v_n^+ \Big) dx = \dfrac{\|(u_n^+, v_n^+)\|^2}{2},\quad \int_{\mathbb{R}^N} \Big( \nabla u_n^- \nabla  v_n^- + u_n^- v_n^- \Big) dx = - \dfrac{\|(u_n^-, v_n^-)\|^2}{2}.
	\end{equation*}
	Il follows that
	\begin{equation}
		\label{eq 18}
		\int_{\mathbb{R}^N} \Big( \nabla u_n \nabla \overline{v}_n + u_n\overline{v}_n \Big) dx + \int_{\mathbb{R}^N} \Big(\nabla v_n  \nabla \overline{u}_n +v_n \overline{u}_n \Big)dx = \|(u_n, v_n)\|^2.
	\end{equation} 
	On the other hand, since $|\overline{v}_n|\le |v_n^+|+|v_n^-|$ and $|v_n^+|_2 +|v_n^-|_2\le 2|v_n|_2$, by $(\ref{eq 13})$, we have 
	\begin{eqnarray*}
		\int_{\mathbb{R}^N} \Big(g_1(x)f_1(v_n) \overline{v}_n\Big) dx &\le& \int_{\mathbb{R}^N} |g_1(x)|\Bigg(\dfrac{S}{4|g_1|_\infty} |v_n| +c_0\Bigg)\Bigg(|v_n^+|+|v_n^-|\Bigg)dx\\
		&\le& \int_{\mathbb{R}^N}\Bigg( \dfrac{S}{4}|v_n|+c_0|g_1(x)|\Bigg)\Bigg(|v_n^+|+|v_n^-|\Bigg)dx\\
		&\le& \Bigg(\dfrac{S}{4}|v_n|_2 +c_0|g_1|_2\Bigg)\Bigg(|v_n^+|_2+|v_n^-|_2\Bigg) \quad \text{(by Hölder inequality)}\\
		&\le& \Bigg(\dfrac{S}{4}|v_n|_2 +c_0|g_1|_2\Bigg)\Bigg(2|v_n|_2\Bigg)\\
		&\le& \dfrac{1}{2}\|v_n\|_1^2 +2c_0|g_1|_2S^{-\frac{1}{2}}\|v_n\|_1\quad \text{(by $(\ref{eq 12})$)}.
	\end{eqnarray*}
	Similarly, we have 
	\begin{equation*}
		\int_{\mathbb{R}^N} \Big(g_2(x)f_2(u_n) \overline{u}_n\Big) dx \le \dfrac{1}{2}\|u_n\|_1^2 +2c_0|g_2|_2S^{-\frac{1}{2}}\|u_n\|_1.
	\end{equation*}
	That is, 
	\begin{multline}
		\label{eq 19}
		\int_{\mathbb{R}^N} \Big(g_1(x)f_1(v_n) \overline{v}_n\Big) dx\\ + \int_{\mathbb{R}^N} \Big(g_2(x)f_2(u_n) \overline{u}_n\Big) dx \le \dfrac{1}{2}\Big(\|u_n\|_1^2 +\|v_n\|_1^2\Big) +2 c_0S^{-\frac{1}{2}}\Big(|g_2|_2\|u_n\|_1 +|g_1|_2\|v_n\|_1\Big)\\
		\le \dfrac{1}{2}\|(u_n,v_n)\|^2 +2 c_0S^{-\frac{1}{2}} \Big(|g_1|_2+|g_2|_2\Big)\|(u_n,v_n)\|.
	\end{multline}
	Equations $(\ref{eq 18})$ and $(\ref{eq 19})$, taken together, yield
	\begin{equation*}
		\|(u_n,v_n)\| \ge \|(u_n,v_n)\|^2- \dfrac{1}{2}\|(u_n,v_n)\|^2 -2 c_0S^{-\frac{1}{2}} \Big(|g_1|_2+|g_2|_2\Big)\|(u_n,v_n)\|.
	\end{equation*}
	
	That is, 
	\begin{equation*}
		\|(u_n,v_n)\| \ge \dfrac{1}{2}\|(u_n,v_n)\|^2 -2 c_0S^{-\frac{1}{2}} \Big(|g_1|_2+|g_2|_2\Big)\|(u_n,v_n)\|.
	\end{equation*}
	
	It follows that $\Big\{(u_n,v_n)\Big\}$ is bounded. There exists a subsequence of $\Big\{(u_n,v_n)\Big\}$, still denoted $\Big\{(u_n,v_n)\Big\}$, and there exists $(u,v)\in X$ such that 
	\begin{equation*}
		(u_n,v_n) \to (u,v) \quad \text{weakly in}\quad X, \quad \text{as}\quad n \to \infty, 
	\end{equation*}
	
	\begin{equation*}
		(u_n,v_n) \to (u,v), \quad \text{in}\quad L_{loc}^2(\mathbb{R}^N) \times L_{loc}^2(\mathbb{R}^N),\quad \text{as}\quad n \to \infty,
	\end{equation*}
	
	\begin{equation*}
		\big(u_n(x), v_n(x)\big) \to \big(u(x),v(x)\big), \quad \text{almost everywhere}, \quad \text{in}\quad \mathbb{R}^N, \quad \text{as}\quad n \to \infty.
	\end{equation*}
	
	We deduce that, for every $(w,z)\in \mathcal{C}_c^\infty(\mathbb{R}^N)\times \mathcal{C}_c^\infty(\mathbb{R}^N) $, 
	\begin{equation*}
		J'(u,v)(w,z)=\lim\limits_{n\to \infty}J'(u_n,v_n)(w,z)=0.
	\end{equation*}
	We conclude that $J'(u,v)=0$. That is, $(u,v)$ is a solution of $(\ref{eq 1})$.\\
	The proof of Theorem $\ref{theorem 3.1}$ is thus complete. \quad \hbox{$\square$}

\end{document}